\documentclass{amsart}

\usepackage{amssymb}
\usepackage{amsrefs}
\usepackage[all]{xy}

\numberwithin{equation}{subsection}

\newtheorem{theorem}{Theorem}[subsection]

\begin{document}

\title{equivariant differential characters and symplectic reduction}

\author{Eugene Lerman}
\thanks{Supported in part by NSF grant DMS-0603892}
\author{Anton Malkin}
\thanks{Supported in part by NSF grant DMS-0456714}
\address{Department of Mathematics, 
University of Illinois at Urbana-Champaign, 
1409 W. Green Street, Urbana, IL 61801, USA}

\begin{abstract}
We describe equivariant differential characters 
(classifying equivariant circle bundles with connections),
their prequantization, and reduction.
\end{abstract}

\maketitle

\tableofcontents

\setcounter{section}{-1}

\section{Introduction}

This paper is a continuation of \cite{LermanMalkinPreq}. The goal
of the project is to understand geometric (pre)quantization and 
symplectic reduction in the context of stacks. 
There are two practical reasons for doing that.
First, even if one is only interested in global
Lie group actions on manifolds, symplectic quotients are generally
orbifolds (i.e. Deligne-Mumford stacks). Second, the language of
stacks allows one to work locally and thus to avoid messy
\v{C}ech-type arguments.

\subsection{Prequantization}
The classic geometric prequantization theorem 
(due to Weil \cite{Weil1952} and Kostant \cite{Kostant1970})
says that given a closed integral differential 2-form 
$\omega$ on a manifold $M$
there exists a principal $S^1$-bundle $P \rightarrow M$ 
with a connection $A$ such that 
the curvature of $A$ is equal to $\omega$, and moreover, the set of 
isomorphism classes of such pairs $(P, A)$ is 
a principal homogeneous space of the group of flat $S^1$-bundles.
Recall that a 2-form $\omega$ on a manifold $M$ is integral if 
$\int_S \omega \in \mathbb{Z}$ for any closed smooth singular 2-chain
$S \in Z_2 (M)$. 

It would be preferable for prequantization procedure to produce
unique output (an $S^1$-bundle with connection). 
For example, this is clearly required to make sense of statements like
``prequantization commutes with reduction'' - see below.
In order to have unique output of the prequantization 
one has to refine its input. One way to do it, due to Cheeger and Simons 
\cite{CheegerSimons1985}, is by using 
differential characters (an alternative approach
is provided by Deligne cohomology - see \cite{Beilinson1984}).
A differential character of degree 2 (degree 1 in Cheeger-Simons 
grading) is a pair $(\omega , \chi )$, where 
$\omega \in \Omega^2 (M)$ is a closed 
2-form and $\chi : Z_1 (M) \rightarrow \mathbb{R}/\mathbb{Z}$ is a character 
of the group $Z_1 (M)$ of smooth singular 1-cycles. 
This pair should satisfy the following compatibility condition: 
\begin{equation}\nonumber
\chi ( \partial S )  = \int_{S} \omega 
\quad \mathrm{mod} \ \mathbb{Z} \ ,
\end{equation}
for any smooth singular 2-chain $S \in C_2 (M)$. One can show
(cf. \citelist{\cite{CheegerSimons1985} \cite{HopkinsSinger2005}
\cite{LermanMalkinPreq}} for various versions of the proof)
that differential characters classify isomorphism classes of 
principal $S^1$-bundles with connections. The classifying bijection
(whose inverse is the prequantization map) associates
to a principal $S^1$-bundle $P$ with a connection $A$ 
its differential character 
$(\omega , \chi )$, where $\omega$ is the curvature of $A$ and
$\chi$ is the holonomy of $A$ 
(we identify $S^1$ with $\mathbb{R}/\mathbb{Z}$ throughout the paper).

\subsection{Lie group actions}
Let $G$ be a Lie group, $\mathfrak{g}$ the Lie algebra of $G$, 
$\mathfrak{g}^*$ the dual space of $\mathfrak{g}$, and
$<,>$ the natural pairing between $\mathfrak{g}$ and $\mathfrak{g}^*$.
Suppose $G$ acts on a manifold $M$ and
consider a $G$-equivariant principal $S^1$-bundle $P \rightarrow M$ with an
invariant connection $A$. That means we are given
a lifting of the $G$-action to the total space $P$ 
(commuting with the $S^1$-action) and the connection form 
$A \in \Omega^1 (P)$ is 
$G$-invariant. Then the curvature $\omega$ of $A$ is also $G$-invariant and 
moreover there is a $G$-equivariant map 
$\mu  : P \rightarrow \mathfrak{g}^*$ given by 
$<\! \mu , X \!\! > = i_{\varepsilon_{x}} A$, where $\varepsilon_{X}$ is
the vector field on $P$ generating the action of an element 
$X \in \mathfrak{g}$. The map $\mu$ is $S^1$-invariant and
hence descends to a map from $M$ to $\mathfrak{g}^*$ which we
also denote by $\mu$. It follows from the definition of the curvature that
$i_{\varepsilon_{x}} \omega = - < \! d \mu , X \!\! >$. Hence
$\mu$ is the moment map in the sense of equivariant symplectic geometry
(see for example \cite{GuilleminSternberg1999}).
Note that if $G$-action is free (so that the quotient $M/G$ is a manifold) 
then a $G$-equivariant bundle $P$ with an invariant connection $A$ descends to 
$M/G$ iff the moment map is identically equal to zero, in other words 
iff the forms $A$ and $\omega$ are $G$-basic. Recall that a 
differential form is $G$-\emph{basic} 
if it is $G$-invariant and $G$-horizontal (i.e.  
vanish on vector fields generating the action of $\mathfrak{g}$).

Equivariant prequantization procedure 
should produce a $G$-equivariant principal 
$S^1$-bundle with an invariant connection given a closed invariant 
integral 2-form
$\omega$ on $M$. Obviously there some obstructions. First of all there should
exist a moment map $\mu$ such that 
$i_{\varepsilon_{x}} \omega = - < \! d \mu , X \!\! >$
for any $X \in \mathfrak{g}$. This means that the action is \emph{Hamiltonian}.
If $G$ is connected and simply-connected then existence of the moment map 
ensures existence of an equivariant 
prequantization bundle with an invariant connection (cf. \cite{Kostant1970}).
For a general $G$ there are additional obstructions.
The origin of these complications is that in the equivariant situation 
it is not enough to think of isomorphism classes of bundles -- one needs to
consider an actual bundle in order to lift the $G$-action to the total space.
This becomes much more clear if one uses an equivalent definition
of an equivariant bundle in terms of the action groupoid 
$\xymatrix{{M} &
\ar@<0.3ex>[l]^{p}
\ar@<-0.7ex>[l]_{a}  
{G \times M}}$,
where the two maps are the projection and the action.
In this language an equivariant $S^1$-bundle is a bundle
$P \rightarrow M$ together with an isomorphism 
$\varphi: a^* P \xrightarrow{\sim} p^* P$ of
the two pull-backs of $P$ to $G \times M$. Note that we
have to use a morphism in the category of bundles on 
$G \times M$, so it is not enough to
consider isomorphism classes of bundles.
Therefore in order to understand equivariant prequantization, we
have to upgrade the prequantization map to a functor from some category
of differential characters to the category of principal $S^1$-bundles 
with connections.

\subsection{The stack of bundles and the stack of differential characters}
Let $\mathcal{DBS}^1 (M)$ be the category of 
$S^1$-bundles with connections on a manifold $M$. 
In  \cite{HopkinsSinger2005} Hopkins and Singer introduced a category
of differential characters $\mathcal{DC}_2^2 (M)$ and showed
that it is equivalent to $\mathcal{DBS}^1 (M)$. As the name suggests,
the isomorphism classes of objects of $\mathcal{DC}_2^2 (M)$
are in bijection with differential characters. We give the complete
definition of the category $\mathcal{DC}_2^2 (M)$ in Subsection
\ref{PreqDef}. Let us just mention that an object of 
$\mathcal{DC}_2^2 (M)$ is a triple $(c,h, \omega )$, where
$c \in Z^2 (M, \mathbb{Z})$ is a smooth singular 2-cocycle,
$\omega \in \Omega^2 (M)$ is a closed 2-form, and
$h \in C^1 (M, \mathbb{R})$ is a smooth singular 1-chain satisfying
$dh = c - \omega$. 
In the previous paper \cite{LermanMalkinPreq} we introduced a 
prequantization functor 
\begin{equation}\nonumber
\mathrm{Preq} : \mathcal{DC}_2^2 (M) \to \mathcal{DBS}^1 (M) \ ,
\end{equation}
such that if $\mathrm{Preq} ((c,h, \omega ))$ is a bundle $P$ 
with a connection
$A$, then the Chern class of $P$ is $[c] \in H^2 (M, \mathbb{Z})$
and the curvature of $A$ is $\omega$.
Moreover we show (and it is crucial for our definition of
$\mathrm{Preq}$) that both $\mathcal{DC}_2^2$ and $\mathcal{DBS}^1$ 
are stacks over the category of manifolds and 
$\mathrm{Preq}$ is an equivalence of stacks.
This allowed us to define geometric prequantization of
arbitrary stacks over the category of manifolds, in particular
prequantization of orbifolds and equivariant prequantization
of manifolds with Lie group actions. For example,
given a manifold $M$ with an action of a Lie group $G$ we have an
equivalence functor
\begin{equation}\nonumber
\mathcal{DC}_2^2 (M \leftleftarrows G \times M) 
\simeq \mathcal{DC}_2^2 ([M/G]) 
\xrightarrow{\mathrm{Preq}} \mathcal{DBS}^1 ([M/G]) \simeq
\mathcal{DBS}^1 (M \leftleftarrows G \times M) \ .
\end{equation}
Here $[M/G]$ is the quotient stack. Objects of the category
$\mathcal{DBS}^1 (M \leftleftarrows G \times M)$
are $G$-equivariant principal $S^1$-bundles on $M$ with
$G$-basic connections. The category 
$\mathcal{DC}_2^2 (M \leftleftarrows G \times M)$ of
equivariant differential characters is 
defined in Section \ref{PrequantizationSection}. 

\subsection{Invariant connections}
In the present paper we are interested in
equivariant bundles with $G$-invariant (but not necessarily basic)
connections. These form a category 
$d\mathcal{BS}^1_{\mathrm{inv}} (M \leftleftarrows G \times M)$.
The corresponding category of differential characters is defined
in Section \ref{InvariantSection} and is denoted by
$\mathcal{DC}^2_{\mathrm{inv}} (M \leftleftarrows G \times M)$.
In particular a $G$-invariant differential character 
(i.e. an isomorphism class of objects of the category
$\mathcal{DC}^2_{\mathrm{inv}} (M \leftleftarrows G \times M)$) is a triple 
$(\omega, \mu , \Psi)$, where $\omega \in \Omega^2 ( M )$ is a 2-form, 
$\mu : M \rightarrow \mathfrak{g}^*$ is a (moment) map, and 
$\Psi$ is a character 
$Z_1 (M \leftleftarrows G \times M ) \rightarrow \mathbb{R}/\mathbb{Z}$
of the group $Z_1 (M \leftleftarrows G \times M )$ of smooth singular
$1$-cycles on the action groupoid. We refer the reader to  
Section \ref{CharactersSection} for details, but let us give here two
examples of subgroups of $Z_1 (M \leftleftarrows G \times M )$.
First one is the group $Z_1 (M)$ of 1-cycles on $M$. The restriction
of $\Psi$ to $Z_1 (M)$ together with $\omega$ form a differential
character on $M$ corresponding to a bundle with connection
(without equivariant structure). The second typical subgroup of 
$Z_1 (M \leftleftarrows G \times M )$ is the stabilizer group
(in $G$) of a point on $M$. The restriction of $\Psi$ to this
subgroup corresponds (under prequantization map) to the action of the
stabilizer group on the fiber of the prequantization $S^1$-bundle.
Let us also note that the character is equivariant (in other words,
belongs to $\mathcal{DC}^2_2 (M \leftleftarrows G \times M) \simeq
\mathcal{DC}^2_2 ([M/G])$ iff $\omega$ is basic 
or, equivalently, iff $\mu$ vanishes identically.

Here is the first main result of the present paper 
stated both in terms of categories and isomorphism classes of objects
(see Theorems \ref{GlobalQPreqThm}  and
\ref{ActionCharThm} for more precise/complete statements) .

\begin{theorem}\label{IntroThmA}
Suppose a Lie group $G$ acts on a manifold $M$. 
\begin{itemize}
\item
There is an equivalence
of categories (prequantization functor)
\begin{equation}\nonumber
\mathrm{Preq} : 
\mathcal{DC}_{\mathrm{inv}}^2 (M \leftleftarrows G \times M) 
\xrightarrow{\sim}  
d\mathcal{BS}^1_{\mathrm{inv}} (M \leftleftarrows G \times M) 
\end{equation}
from the category 
$\mathcal{DC}_{\mathrm{inv}}^2 (M \leftleftarrows G \times M)$
of $G$-invariant differential characters 
to the category 
$d\mathcal{BS}^1_{\mathrm{inv}} (M \leftleftarrows G \times M)$
of $G$-equivariant principal $S^1$-bundles with $G$-invariant connections.
\item
Isomorphism classes of $G$-equivariant $S^1$-bundles 
with $G$-invariant connections are in bijection
with $G$-invariant differential characters, i.e. triples 
$(\omega, \mu , \Psi)$, where $\omega \in \Omega^2 ( M )$ 
is a $G$-invariant  2-form (the curvature of the connection), 
$\mu : M \rightarrow \mathfrak{g}^*$ is a $G$-equivariant (moment) map, 
and $\Psi$ is a $G$-invariant character 
$Z_1 (M \leftleftarrows G \times M ) \rightarrow \mathbb{R}/\mathbb{Z}$.
\end{itemize}
\end{theorem}

\subsection{Symplectic reduction} 
We continue with our setup of a Lie group $G$ acting on a manifold $M$,
a $G$-equivariant principal $S^1$-bundle $P \rightarrow M$ with 
a $G$-invariant connection $A \in \Omega^1(P)$, 
the curvature $\omega \in \Omega^1 (M)$, and the moment map 
$\mu: M \rightarrow \mathfrak{g}^*$. Assume zero is a regular value
of the moment map. Then $\mu^{-1} (0)$ is a manifold and the restriction
of the pair $(P, A)$ to $\mu^{-1} (0)$ is a principal $S^1$-bundle
with a basic connection, hence a bundle with connection 
on the stack $[\mu^{-1} (0)/G]$. 
On the other hand one can restrict the corresponding
differential character $(\omega , \mu , \Psi )$ to $\mu^{-1} (0)$
and the restricted character descends 
to the quotient stack $[\mu^{-1} (0) /G]$. 
In particular the form $\omega|_{\mu^{-1} (0)}$
descends to $[\mu^{-1} (0) /G]$, which is the point of symplectic reduction 
(to be honest, a deeper statement is about 
non-degeneracy of the reduction of a non-degenerate 2-form $\omega$ 
\citelist{\cite{MarsdenWeinstein1974} \cite{Meyer1973}},
but the non-degeneracy condition is not relevant for geometric 
(pre)quantization).

Our second result is that the prequantization bijection between
bundles with connections
and differential characters commutes with reduction. Since we are interested
in equivalences of categories let us change the point of view and 
consider all bundles/characters such that their moment map vanishes 
on a given $G$-stable submanifold $N \subset M$. 
We denote the corresponding subcategories by  
$d\mathcal{BS}^1_{\mathrm{inv}, N} (M \leftleftarrows G \times M)$ and
$\mathcal{DC}_{\mathrm{inv}, N}^2 (M \leftleftarrows G \times M)$.
Then we have the following result (see Theorem \ref{ActionReductionThm}).
\begin{theorem}\label{IntroThmB}
Let a Lie group $G$ act on a manifold $M$ and
$i: N \hookrightarrow M$ an inclusion of a $G$-stable submanifold.
The diagram
\begin{equation}\nonumber
\xymatrix{
{\mathcal{DC}^2_{\mathrm{inv}} (M \leftleftarrows G \times M)} 
\ar[r]^{\mathrm{Preq}} 
&
{d\mathcal{BS}^1_{\mathrm{inv}} (M \leftleftarrows G \times M)}
\\
{\mathcal{DC}^2_{\mathrm{inv}, N} (M \leftleftarrows G \times M)} 
\ar@{^{(}->}[u]
\ar[r]^{\mathrm{Preq}} 
\ar[d]^{\textstyle i^*}
&
{d\mathcal{BS}^1_{\mathrm{inv}, N} (M \leftleftarrows G \times M)}
\ar@{^{(}->}[u]
\ar[d]^{\textstyle i^*}
\\
{\mathcal{DC}^2_{2} (N \leftleftarrows G \times N)} 
\ar[r]^{\mathrm{Preq}}
\ar[d]^{\textstyle \wr}
&
{\mathcal{DBS}^1 (N \leftleftarrows G \times N)}
\ar[d]^{\textstyle \wr}
\\
{\mathcal{DC}^2_{2} ([ N / G ])} 
\ar[r]^{\mathrm{Preq}}
&
{\mathcal{DBS}^1 ([ N / G ])}
}
\end{equation}
commutes up to natural transformations. 
Here the first row represents $G$-invariant characters/connections,
the second row has moments vanishing on the submanifold $N$, the third is
the restriction of the second to $N$, and the fourth is
the descent of the third to the quotient (which makes sense 
because of the moment vanishing).
\end{theorem}

\subsection{General Lie groupoids}
The above discussion concerned bundles, connections, and differential 
characters, in the setup of a Lie group $G$ acting on a manifold $M$. 
However it should be clear by now that definitions and results are best stated
using the language of an action groupoid $M \leftleftarrows G \times M$.
Hence it makes sense to work in the context of a general Lie groupoid
$M \leftleftarrows \Gamma$,
which we do most of the time in the present paper. Our other motivation
for using general groupoids is in that they provide atlases for stacks 
(and in particular for manifolds). In other words, we try to describe
prequantization/reduction locally. One should be careful however with
stacky interpretations of groupoid results.
Not every constructions is atlas-independent (i.e. makes sense for 
the underlying stack). For example, basic connections descend
to the quotient stack while invariant ones don't. 

\subsection{Structure of the paper}
Our goal is to describe equivariant bundles with invariant connections
in terms of differential characters. 
We do this in two steps. First, in Section \ref{PrequantizationSection},
we define a category $d\mathcal{BS}^1 (M \leftleftarrows \Gamma)$ 
of equivariant principal $S^1$-bundles with arbitrary connection. 
Roughly speaking, this category sits in between the category 
$\mathcal{BS}^1 (M \leftleftarrows \Gamma)$ 
of equivariant bundles without connections and the category
$\mathcal{DBS}^1 (M \leftleftarrows \Gamma)$ of equivariant bundles with
basic connections. 
Then in Theorem \ref{InvPreq} we construct a (prequantizaion) equivalence
\begin{equation}\nonumber
\mathrm{Preq} : \mathcal{DC}^2_{2-1} (M \leftleftarrows \Gamma)
\xrightarrow{\sim} d\mathcal{BS}^1 (M \leftleftarrows \Gamma) \ ,
\end{equation}
where $\mathcal{DC}^2_{2-1} (M \leftleftarrows \Gamma)$ is certain
category of differential characters sitting between
the category $\mathcal{DC}^2_{1} (M \leftleftarrows \Gamma)$
classifying equivariant principal $S^1$-bundles and the category
$\mathcal{DC}^2_{2} (M \leftleftarrows \Gamma)$
classifying equivariant bundles with basic connections.
This is the main technical result of the paper. 
Our construction of the functor $\mathrm{Preq}$ has two important features.
First, it is explicit enough to allow us 
to track invariance properties of the connections via
the corresponding differential characters. More precisely,
a connection on an equivariant
bundle canonically determines a 1-form $\alpha \in \Omega^1 (\Gamma)$
and this form is a part of the differential character
(see Theorem \ref{InvPreq} for details). Second property of 
$\mathrm{Preq}$ is that it commutes 
(up to natural transformations) with pull-backs (in particular,
restrictions). This implies the ``prequantization commutes with
reduction'' theorem.

On the second step, in Section 3, we translate invariance property 
of connections into conditions on the corresponding differential
character and thus prove the first part of Theorem \ref{IntroThmA}.
We also discuss the moment map and invariant connections/curvatures 
on a general groupoid. In particular, we observe that the invariance condition 
only makes sense on the zero level of the moment map. So 
in the groupoid case the usual (global action)
reduction procedure: consider invariant symplectic forms -- restrict to
the zero level of the moment map -- 
take quotient, should be reversed as follows:
restrict to the zero level of the moment map -- consider invariant forms --
take quotient (in fact, the quotient is a tautology in the stacky language).

In Section \ref{ReductionSection} we discuss reduction and
prove Theorem \ref{IntroThmB}.

In Section \ref{CharactersSection} we provide an explicit description 
of the actual differential characters on a groupoid 
(i.e of the isomorphism classes of
objects of $\mathcal{DC}^2_{2-1} (M \leftleftarrows \Gamma)$) and
prove the second part of Theorem \ref{IntroThmA}.

Finally, Section \ref{ExamplesSection} provides some examples of
differential characters and reduction.

\subsection{Previous Results}
There are several recent papers dealing with equivariant
geometric prequantization
and reduction. In the case of a global Lie group action 
Gomi \cite{Gomi2005} describes prequantization via 
equivariant Deligne cohomology. Lupercio and Uribe 
\cite{LupercioUribe2006} describe equivariant differential characters 
for a finite group action (their construction is different from ours).
Behrend and Xu \cite{BehrendXu2006} prove a version of
Kostant's prequantization theorem on proper groupoids
(we explain how to derive a generalization of their result using differential
characters in subsection \ref{KostantSubSection}).  

Besides being restricted to global quotients 
(except \cite{BehrendXu2006}) the above papers deal with
isomorphism classes rather than stacks of bundles. As explained 
above we believe that it is impossible to properly understand geometric
quantization on stacks without working with categories of bundles
(i.e. with both objects and morphisms).

Bos \cite{Bos2007} deals with symplectic reduction in a Lie groupoid case.
His approach (and in fact his problem) is different from ours --
he studies certain \emph{internal} symplectic structure/reduction.

As always when talking about symplectic structures and groupoids one should
not confuse symplectic structures on the manifolds of objects (our situation)
and on arrows (``symplectic groupoids'' of Weinstein \cite{Weinstein1987}).

\subsection{Higher degree versions}
In this paper we are dealing with invariant 
differential characters of degree 2. 
Similar results hold for degree 3 characters (classifying
$S^1$-gerbes with connective structures). In particular
equivariant string connections appear naturally in this approach.

We wanted to keep this text free of obscure notions of 
higher stacks and their descent conditions and so postponed
discussion of gerbes and their characters for another paper.

\section{Circle bundles with connections and  prequantization}
\label{PrequantizationSection}

\subsection{Prequantization}\label{PreqDef}
We recall the construction of the prequantization
functor in \cite{LermanMalkinPreq}. 
Let $M$ be a differentiable manifold and consider the category
$\mathcal{DBS}^1 (M)$ of principal $S^1$-bundles with connections on $M$,
\emph{i.e.} the objects of $\mathcal{DBS}^1 (M)$ are $S^1$-bundles
with connections and morphisms are smooth $S^1$-equivariant maps preserving
connections. In \cite{LermanMalkinPreq} we described a
prequantization functor
\begin{equation}\nonumber
\mathrm{Preq} : \mathcal{DC}_2^2 (M) \to \mathcal{DBS}^1 (M) \ ,
\end{equation}
where $\mathcal{DC}^2_2 (M)$ is the category of 
differential characters defined below. We also showed that
$\mathrm{Preq}$ is an equivalence of categories. Thus one can say that
differential characters classify $S^1$-bundles with connections.

To define the category of differential characters 
(following Hopkins-Singer \cite{HopkinsSinger2005}) 
we start with a complex $DC^{\bullet}_s (M)$ of abelian groups:
\begin{equation}\nonumber
DC^n_s (M) = \{ ( c, h, \omega ) \ | \ \omega = 0 \text{ if } n < s \} 
\subset C^n (M, \mathbb{Z}) \times C^{n-1} (M, \mathbb{R}) 
\times \Omega^n (M) \ ,
\end{equation}
where $C^n (M, R)$ denotes the group of smooth
singular cochains with coefficients in a ring $R$, 
$\Omega^n (M)$ denotes the group of differential forms,
and $s$ is a fixed positive integer (truncation degree).
The differential in the complex $DC^{\bullet}_s (M)$ is given by
\begin{equation}\nonumber
d ( c, h, \omega ) = 
( dc \, ,\  \omega - c- dh \, ,\ d\omega ) \ .
\end{equation}
The group of  Cheeger-Simons differential characters 
\cite{CheegerSimons1985} of degree $2$ on $M$ 
is isomorphic to $H^2 (DC^\bullet_2 (M))$.

The category $\mathcal{DC}^2_2 (M)$ is defined as follows:
\begin{itemize}
\item
Objects are $2$-cocycles: $(c, h, \omega) \in DC^2_2 (M)$, 
$d (c, h, \omega)=0$
\item
Morphisms are 1-cochains up to exact 1-cochains. The set of morphisms
from $(c, h, \omega)$ to $(c', h', \omega')$ is   
\begin{equation}\nonumber
\frac{\{ (a, t, 0) \in DC^1_2 (M) \ | \ d(a, t, 0) = 
(c, h, \omega) - (c', h', \omega')\}}
{(a, t, 0) \sim (a, t, 0) + d (m, 0, 0) 
\text{ for } (m, 0, 0) \in DC^0_2 (M)}  \ .
\end{equation}
\item
Composition of morphisms is the addition in $DC^1_2 (M)$.
\end{itemize}

Note that given a smooth map $\rho: M \rightarrow N$, one can
define natural pull-back functors
$\mathcal{DC}^2_2 ( \rho ): 
\mathcal{DC}^2_2 (N) \rightarrow \mathcal{DC}^2_2 (M)$
and 
$\mathcal{DBS}^1 ( \rho ): 
\mathcal{DBS}^1 (N) \rightarrow \mathcal{DBS}^1 (M)$.
This means that $\mathcal{DC}^2_2$ and $\mathcal{DBS}^1$ 
form presheaves of groupoids over the category of manifolds, 
and it is shown in \cite{LermanMalkinPreq} that these presheaves
are in fact stacks. We don't use stacky language in the present paper
but it is needed for the proof of the prequantization theorem below.

Now we are ready to define the prequantization functor.
Let $\mathcal{DBS}^1_{triv} (M)$ denote the full subcategory
of $\mathcal{DBS}^1 (M)$ consisting of trivial bundles with
arbitrary connections and consider a functor  
$\mathrm{DCh}_{\mathrm{triv}} : 
\mathcal{DBS}_{\mathrm{triv}}^1 \rightarrow \mathcal{DC}_2^2$
given by:
\begin{itemize}
\item
$\mathrm{DCh}_{\mathrm{triv}} ((M \times S^1 , \beta + d\theta )) =
(0, \beta, d \beta)$ on objects. Here $\beta \in \Omega^1 (M)$, 
$\theta \in \mathbb{R}/\mathbb{Z}$ is the
coordinate on $S^1$, and $\beta+d\theta$ is a connection 1-form on
the trivial bundle $M \times S^1 \rightarrow M$. 
\item
$\mathrm{DCh}_{\mathrm{triv}} (f)= [(d(\tilde{f}-f), -\tilde{f}, 0)] 
\in DC^1_2 / d (DC^0_2)$ on morphisms.
Here we think of a morphism in $\mathcal{DBS}_{\mathrm{triv}}^1$ from 
$(M \times S^1 , a + d\theta )$ to $(M \times S^1 , a' + d\theta )$
as a smooth function
$f: M \rightarrow S^1 = \mathbb{R}/\mathbb{Z}$ such that 
$df = a' - a$ and let $\tilde{f} \in C^0 (M, \mathbb{R})$ be 
a lift of $f$. 
\end{itemize}
One of the main results of \cite{LermanMalkinPreq} is:
\begin{theorem}\label{OldPreqThm}There is unique (up to
a natural transformation) equivalence of stacks
\begin{equation}\nonumber
\mathrm{DCh} : \mathcal{DBS}^1 \rightarrow \mathcal{DC}_2^2 \ , 
\end{equation}
extending $\mathrm{DCh}_{\mathrm{triv}}$. In other words, there is 
a family of equivalence functors (one for each manifold):
\begin{equation}\nonumber
\mathrm{DCh} (M) : \mathcal{DBS}^1 (M) \to \mathcal{DC}_2^2 (M) 
\end{equation}
such that
(i) $\mathrm{DCh}$ commutes (up to coherent natural transformations) 
with pull-backs, and 
(ii) on trivial bundles with connections $\mathrm{DCh}$ is equal
to $\mathrm{DCh}_{\mathrm{triv}}$.

We call a quasi-inverse functor of $\mathrm{DCh}$ 
prequantization functor and denote it $\mathrm{Preq}$
(it is defined up to a natural transformation).
\end{theorem}
The idea behind this theorem is that any bundle becomes
trivial after pull-back to a contractible open subset of $M$.
A similar argument (cf. \cite{LermanMalkinPreq}) 
shows the category $\mathcal{DC}^2_1 (M)$
is equivalent to the category $\mathcal{BS}^1$ of principal $S^1$-bundles
(without connections).

\subsection{Equivariant $S^1$-bundles on groupoids} 
Consider a Lie groupoid $M \leftleftarrows \Gamma$. 
Its nerve is the simplicial manifold
\begin{equation}\nonumber
\xymatrix{
{M} &
\ar@<0.3ex>[l]^{\partial_0 =s}
\ar@<-0.7ex>[l]_{\partial_1 =t}  
{\Gamma} &
\ar@<0.9ex>[l]^{\partial_0}
\ar@<-0.1ex>[l]
\ar@<-1.1ex>[l]_{\partial_2}  
{\Gamma_2} &
\ar@<1.3ex>[l]^{\partial_0}
\ar@<0.3ex>[l]
\ar@<-0.7ex>[l]
\ar@<-1.7ex>[l]_{\partial_3}  
{\cdots}}
\end{equation}
where
\begin{equation}\nonumber
\Gamma_n = \underbrace{
\Gamma \times_{_M} \Gamma \times_{_M} \ldots \times_{_M} \Gamma}_n \ .
\end{equation}
The maps 
$\partial_0 = s$, $\partial_1 =t :  \Gamma \rightarrow M$ are the source and 
the target maps of the groupoid respectively. We use
notations $(s,t)$ and $(\partial_0, \partial_1)$ interchangeably
throughout the paper. The maps 
$\partial_0$, $\partial_n :\Gamma_n \rightarrow \Gamma_{n-1}$
forget the first and the last $\Gamma$-factors respectively, and 
finally, $\partial_i :\Gamma_n \rightarrow \Gamma_{n-1}$ for
$1 \leq i \leq n-1$ multiplies (composes) $i^{th}$ and $(i+1)^{th}$ factors.
We omit face maps $\Gamma_{n-1} \rightarrow \Gamma_{n}$ (in particular
the unit map for the groupoid) in the above diagram. 

A principal $S^1$-bundle $P \rightarrow M$ is 
$\Gamma$-equivariant if one can lift it to a simplicial principal 
$S^1$-bundle
\begin{equation}\nonumber
\xymatrix{ 
\ar[d]_{\pi^{_M}} {P} &
\ar@<-0.0ex>[l]^{\partial^P_0}
\ar@<-0.8ex>[l]_{\partial^P_1}  
\ar[d]^{\pi^{_\Gamma}} {P_1} &
\ar@<0.9ex>[l] \ar@<-0.1ex>[l] \ar@<-1.1ex>[l] 
\ar[d]^{\pi^{_{\Gamma_2}}}
{P_2} &
\ar@<1.3ex>[l]
\ar@<0.3ex>[l]
\ar@<-0.7ex>[l]
\ar@<-1.7ex>[l]  
{\cdots} \\
{M} &
\ar@<0.0ex>[l]^{\partial^M_0}
\ar@<-0.8ex>[l]_{\partial^M_1}  
{\Gamma} &
\ar@<0.9ex>[l]
\ar@<-0.1ex>[l]
\ar@<-1.1ex>[l]  
{\Gamma_2} &
\ar@<1.3ex>[l]
\ar@<0.3ex>[l]
\ar@<-0.7ex>[l]
\ar@<-1.7ex>[l]  
{\cdots}}
\end{equation}
More precisely, such a lift provides an equivariant structure
on the bundle. An equivalent definition of an equivariant
structure on $P$ is a bundle isomorphism 
$\varphi: (\partial^M_1)^{\ast} P \xrightarrow{\sim} (\partial^M_0)^{\ast} P$
satisfying a cocycle condition on $\Gamma_2$. 
In the case of an action
groupoid $M \leftleftarrows G \times M$ an equivariant structure 
on a bundle $P \rightarrow M$ is equivalent to a lift of the $G$-action on M
to an action on the total space $P$.

\subsection{Connections on equivariant $S^1$-bundles} 
Now suppose we are given a connection $A$ on the bundle $P \rightarrow M$.
We think of $A$ as a 1-form on $P$. The connection
$A$ is $\Gamma$-basic (with respect to a given equivariant
structure on $P$) if $(\partial_1^P)^{\ast} A = (\partial_0^P)^{\ast} A$.
In the case of an action groupoid $M \leftleftarrows G \times M$
a connection is basic if it is $G$-invariant and vanish on vector
fields generating the infinitesimal action of the Lie algebra
$\mathfrak{g}$ of $G$. 
Equivariant $S^1$-bundles with basic connections
form a category $\mathcal{DBS}^1(M \leftleftarrows \Gamma )$
with morphisms being equivariant bundle maps preserving connections. 
In fact (cf. \cite{LermanMalkinPreq}) this category 
(up to an equivalence) depends only on the stack $[M / \Gamma]$,
not on the actual (atlas) groupoid $M \leftleftarrows \Gamma$. 
In other words, it is Morita-equivariant. 

In the present paper we are interested
in arbitrary (not necessarily basic) connections on
an equivariant $S^1$-bundle. So we introduce a category
\mbox{$d\mathcal{BS}^1 (M \leftleftarrows \Gamma )$} of
$\Gamma$-equivariant bundles on $M$ with arbitrary connections. 
Morphisms in \mbox{$d\mathcal{BS}^1 (M \leftleftarrows \Gamma )$} are
equivariant bundle maps preserving connections.
Note that $\mathcal{DBS}^1(M \leftleftarrows \Gamma )$
is a full subcategory of $d\mathcal{BS}^1(M \leftleftarrows \Gamma )$.

\subsection{Singular and de Rham cohomology of groupoids}
Let us recall the definitions of singular and de Rham complexes 
on a groupoid $M \leftleftarrows \Gamma$. We refer the reader
to \cite{Behrend2004} for details. The de Rham complex 
$\Omega^{\bullet} (M \leftleftarrows \Gamma)$ is the total 
complex of the following double complex 
(denoted by $\Omega^{\bullet} (\Gamma_{\bullet})):$
\begin{equation}\nonumber
\xymatrix{
{\vdots} & {\vdots} & \\
{\Omega^1 ( M )} 
\ar[r]^-{\delta} \ar[u]^{d} &
{\Omega^1 ( \Gamma )} 
\ar[r]^-{\delta} \ar[u]^{-d} &
{\cdots}
\\
{\Omega^0 ( M )} 
\ar[r]^-{\delta} \ar[u]^{d} &
{\Omega^0 ( \Gamma )}
\ar[r]^-{\delta} \ar[u]^{-d} &
{\cdots}
\\  
{M} &
\ar@<0.3ex>[l]^{\partial_0}
\ar@<-0.7ex>[l]_{\partial_1}  
{\Gamma} &
\ar@<0.9ex>[l]^{\partial_0}
\ar@<-0.1ex>[l]
\ar@<-1.1ex>[l]_{\partial_2}    
{\cdots}}
\end{equation}
where $\delta$ is the simplicial differential: 
$\delta=\sum (-1)^i \partial^{\ast}_i$. Analogously, given a ring $R$, 
the smooth singular cochain complex $C^{\bullet} (M \leftleftarrows \Gamma , R)$
is defined as the total complex of the double complex
$C^{\bullet} (\Gamma_{\bullet}, R)$. 
One can also define singular chains complex in a similar fashion
(with arrows in the opposite direction). Smooth singular chains are dual to
smooth singular cochains. The usual de Rham Theorem
on manifolds implies (by a spectral sequence argument)
de Rham Theorem on groupoids: 
$H^n (C^{\bullet} (M \leftleftarrows \Gamma , \mathbb{R}))$
is isomorphic to 
$H^n (\Omega^{\bullet} (M \leftleftarrows \Gamma ))$.
The isomorphism is provided by integrating forms over smooth chains.

\subsection{Differential characters on groupoids}
\label{CharDef}
Similarly to singular and de Rham cochain complexes on a groupoid
$M \leftleftarrows \Gamma$ one can define the complex of
differential characters $DC_s^{\bullet} (M \leftleftarrows \Gamma )$
as the total complex of the double complex
$DC_s^{\bullet} (\Gamma_{\bullet})$. It is shown in
\cite{LermanMalkinPreq} that, provided $s > 0$, the total
cohomology of  $DC_s^{\bullet} (\Gamma_{\bullet})$ is equal to
the total cohomology of its truncated version

\begin{equation}\nonumber
\xymatrix{
{\vdots} & {\vdots} & \\
{DC_s^2 ( M )} 
\ar[r]^-{\delta} \ar[u]^{d} &
{DC_s^2 ( \Gamma )} 
\ar[r]^-{\delta} \ar[u]^{-d} &
{\cdots}
\\
{\frac{DC_s^1 ( M )}{d \ DC_s^0 ( M )}} 
\ar[r]^-{\delta} \ar[u]^{d} &
{\frac{DC_s^1 ( \Gamma )}{d \ DC_s^0 ( \Gamma )}}
\ar[r]^-{\delta} \ar[u]^{-d} &
{\cdots}}
\end{equation}
which we'll use from now on.
Translating complexes into categories,
we get a category 
$\mathcal{DC}^2_s (M \leftleftarrows \Gamma)$
associated to the second total cohomology
of the above (truncated) complex 
(2-cocycles are objects, 1-cochains 
modulo exact 1-cochains are morphisms). 
A standard stacky argument (cf. \cite{LermanMalkinPreq}) allows one
to extend the prequantization functor (on manifolds) 
to an equivalence 
\begin{equation}\nonumber
\mathrm{Preq}: \ \mathcal{DC}^2_2 (M \leftleftarrows \Gamma)
\xrightarrow{\sim} \mathcal{DBS}^1 (M \leftleftarrows \Gamma) 
\ .
\end{equation}
Similarly $\mathcal{DC}^2_1 (M \leftleftarrows \Gamma)$
is equivalent to the category $\mathcal{BS}^1 (M \leftleftarrows \Gamma)$
of equivariant $S^1$-bundles.
However presently we are interested in the intermediate category 
$d\mathcal{BS}^1 (M \leftleftarrows \Gamma )$ of
equivariant bundles with arbitrary connections. So we consider a
double complex 
\begin{equation}\nonumber
\xymatrix{
{\vdots} & {\vdots} & \\
{DC_2^2 ( M )} 
\ar[r]^-{\delta} \ar[u]^{d} &
{DC_1^2 ( \Gamma )} 
\ar[r]^-{\delta} \ar[u]^{-d} &
{\cdots}
\\
{\frac{DC_2^1 ( M )}{d \ DC_2^0 ( M )}} 
\ar[r]^-{\delta} \ar[u]^{d} &
{\frac{DC_1^1 ( \Gamma )}{d \ DC_1^0 ( \Gamma )}}
\ar[r]^-{\delta} \ar[u]^{-d} &
{\cdots}}
\end{equation}
Note that truncation degree for differential forms is 2
in the first column and 1 in the second. This corresponds to the fact
that equivariant structure on our bundles (i.e. isomorphism of
the two pull-backs to $\Gamma$) does not necessarily preserve 
connection. Roughly speaking, we would like our characters to
correspond to bundles with connections on $M$ and bundles without
connections on $\Gamma$.
We denote the above complex $DC^2_{2-1} (M \leftleftarrows \Gamma)$
and the corresponding category of total degree-2 cochains 
$\mathcal{DC}^2_{2-1} (M \leftleftarrows \Gamma)$.
The complex $DC^2_{2-1} (M \leftleftarrows \Gamma)$ 
contains  
$DC^2_2 (M \leftleftarrows \Gamma)$ as a subcomplex,
and the category $\mathcal{DC}^2_{2-1} (M \leftleftarrows \Gamma)$ 
contains  
$\mathcal{DC}^2_2 (M \leftleftarrows \Gamma)$ as a full subcategory.
Explicitly, a degree-2 cochain in 
$DC^2_{2-1} (M \leftleftarrows \Gamma)$ 
(i.e an object of 
\mbox{$\mathcal{DC}^2_{2-1} (M \leftleftarrows \Gamma)$)}
looks like
\begin{equation}\nonumber
\xymatrix{
{(c, h, \omega )} &
{\bullet} & 
\\
{\bullet} &
{[(b, f, \alpha )]} & 
}
\end{equation}
where $c \in C^2 (M, \mathbb{Z})$, 
$b \in C^1 (\Gamma , \mathbb{Z})$,  
$h \in C^1 (M, \mathbb{R})$, 
$f \in C^0 (\Gamma , \mathbb{R})$, 
$\omega \in \Omega^2 (M )$, 
$\alpha \in \Omega^1 (\Gamma )$, and
$[(b,f,\alpha )]$ means an equivalence class with respect to
the relation 
$(b,f,\alpha ) \sim (b + dn, f - n , \alpha )$,
$n \in C^0 (\Gamma , \mathbb{Z})$.
We write this cochain as 
$((c, h, \omega ),[(b,f,\alpha )])$. It belongs to
$DC^2_2 (M \leftleftarrows \Gamma)$ iff $\alpha=0$. 

A degree-1 cochain in 
$DC^2_{2-1} (M \leftleftarrows \Gamma)$ 
(i.e a morphism in $\mathcal{DC}^2_{2-1} (M \leftleftarrows \Gamma)$)
looks like
\begin{equation}\nonumber
\xymatrix{
{\bullet} &
{\bullet} & 
\\
{[(a, t, 0 )]} &
{\bullet} & 
} 
\end{equation}
with $a \in C^1 (M, \mathbb{Z})$, 
$t \in C^1 (M, \mathbb{R})$, and
$(a, t, 0 ) \sim (a + dm, t - m , 0 )$ for
$m \in C^0 (M , \mathbb{Z})$.

The first result of the present paper is the following 
prequantization theorem on groupoids:

\begin{theorem}\label{InvPreq}
The prequantization functor
\begin{equation}\nonumber
\mathrm{Preq}: \ \mathcal{DC}^2_2 (M \leftleftarrows \Gamma)
\xrightarrow{\sim} \mathcal{DBS}^1 (M \leftleftarrows \Gamma)
\end{equation}
extends to an equivalence of categories
\begin{equation}\nonumber
\mathrm{Preq}: \ \mathcal{DC}^2_{2-1} (M \leftleftarrows \Gamma)
\xrightarrow{\sim} d\mathcal{BS}^1 (M \leftleftarrows \Gamma)
\end{equation}
such that if 
\begin{equation}\nonumber
\mathrm{Preq} ((c, h, \omega ), [(b, f, \alpha ]) = 
(P \stackrel{\pi}{\rightarrow} M,A)
\end{equation}
then $\omega$ is the curvature of the connection $A$, while
$\alpha$ is the difference of the two pull-backs of $A$ to $P_1$
considered as a 1-form on $\Gamma:$
\begin{equation}\nonumber
\pi^* \alpha = \partial_1^* A - \partial_0^* A \ .
\end{equation}
In particular the second total cohomology group
of the complex 
$DC^{\bullet}_{2-1} (M \leftleftarrows \Gamma)$ 
classifies equivariant principal $S^1$-bundles on
$M \leftleftarrows \Gamma$ with arbitrary connections.
\end{theorem}
\begin{proof}
We give an explicit construction of the quasi-inverse functor
\begin{equation}\nonumber
\mathrm{DCh}: \ d\mathcal{BS}^1 (M \leftleftarrows \Gamma)
\rightarrow \mathcal{DC}^2_{2-1} (M \leftleftarrows \Gamma)
\end{equation}
of the extended prequantization functor $\mathrm{Preq}$. 
Let us begin with objects. Given an 
equivariant bundle $P \rightarrow M$ with a connection
$A \in \Omega^1 (P)$ we apply $\mathrm{DCh}$-functor on $M$
(cf. Subsection \ref{OldPreqThm}) to get a cochain 
$(c,h, \omega) = \mathrm{DCh} \bigl( (P, A) \bigr) \in DC^2_2 (M)$ 
such that
\begin{equation}\label{dCDW}
d(c, h, \omega) = 0 \ .
\end{equation}
Now the equivariant 
structure on $P$ is a bundle isomorphism 
$\partial_1^* P \xrightarrow{\sim} \partial_0^* P$
or a trivialization 
$\sigma: \Gamma \times S^1 \xrightarrow{\sim} 
\partial_1^* P \otimes (\partial_0^* P)^{-1}$.
We put $\sigma^* ( \partial_1^* A - \partial_0^* A) = \alpha + d \theta$
where $\alpha \in \Omega^1 (\Gamma)$.
Applying $\mathrm{DCh}$-functor on $\Gamma$ 
to the isomorphism $\sigma$ and using the 
explicit form of $\mathrm{DCh}$ for trivial bundles with connections 
(cf. Subsection \ref{PreqDef}) we get (a class of) cochain 
$\mathrm{DCh} (\sigma )=[(b, f, 0)] \in  DC^1_2 ( \Gamma )/DC^1_2 ( \Gamma )$ 
such that  
\begin{equation}\label{DCiso1}
d(b, f, 0) = (\partial_1^* (c, h, \omega) - \partial_0^* (c, h, \omega))
-(0, \alpha, d \alpha) \ .
\end{equation}
Note that $[(b, f, 0)]$ encodes not 
only the trivialization $\sigma$ but also natural transformations 
coming from the fact that the functor $\mathrm{DCh}$ commutes only weakly
with pull-backs and products. In any case, \eqref{DCiso1} means that
\begin{equation}\label{dCDWdBUA}
d (b, f, \alpha ) = \delta (c, h, \omega) \ .
\end{equation} 
Since $\sigma$ satisfies the cocycle condition on $\Gamma_2$ 
and $\mathrm{DCh}$ commutes strongly with pull-backs on the level of morphisms, 
we get 
\begin{equation}\label{dBUA}
\delta [(b, f, \alpha )] = 0 \ .
\end{equation} 
Equations \eqref{dCDW}, \eqref{dCDWdBUA}, \eqref{dBUA}, mean that 
$((c, h, \omega), [(b, f, \alpha )])$ is a cocycle in
$\mathcal{DC}^2_{2-1} (M \leftleftarrows \Gamma)$ and we put 
\begin{equation}\nonumber
\mathrm{DCh} ((P, A, \sigma )) = ((c, h, \omega), [(b, f, \alpha )])
\end{equation}
This completes construction of $\mathrm{DCh}$ on objects
of $d\mathcal{BS}^1 (M \leftleftarrows \Gamma)$. 
Note that the differential form parts $\omega$ and $\alpha$ of the cocycle 
$((c, h, \omega), [(b, f, \alpha ))])$ correspond precisely 
to their description in the theorem. In particular
$A$ is basic iff $\alpha = 0$. Hence the restriction
of $\mathrm{DCh}$ to the subcategory 
$\mathcal{DBS}^1 \subset d\mathcal{BS}^1$ takes values in
$\mathcal{DC}_2^2 \subset d\mathcal{DC}_{2-1}^2$ and
coincides with the equivariant $\mathrm{DCh}$ of 
\cite{LermanMalkinPreq}.

Since morphisms in $\mathcal{DC}^2_{2-1} (M \leftleftarrows \Gamma)$
(resp. $d\mathcal{BS}^1 (M \leftleftarrows \Gamma)$) are the same as
morphisms in $\mathcal{DC}^2_{2} (M)$
(resp. $\mathcal{DBS}^1 (M)$) we just use the original
functor $\mathrm{DCh}$ on $M$ for morphisms.

The above construction is clearly reversible 
(using $\mathrm{Preq}$-functors on $M$ and $\Gamma$ as
quasi-inverses of $\mathrm{DCh}$-functors). 
Hence the functor 
\begin{equation}\nonumber
\mathrm{DCh}: \ d\mathcal{BS}^1 (M \leftleftarrows \Gamma)
\rightarrow \mathcal{DC}^2_{2-1} (M \leftleftarrows \Gamma)
\end{equation}
is an equivalence of categories and we denote the quasi-inverse functor 
$\mathrm{Preq}$. Note that morphisms in the categories
$d\mathcal{BS}^1 (M \leftleftarrows \Gamma)$
and $\mathcal{DC}^2_{2-1} (M \leftleftarrows \Gamma)$
change neither the connection $A$ nor differential form
parts $\omega$ and $\alpha$ of the cocycle. Therefore the condition
of the theorem is satisfied for any choice of the quasi-inverse 
functor.
\end{proof}

\subsection{Equivariant Kostant Theorem}\label{KostantSubSection}
Recall that Kostant's prequantization theorem on a manifold $M$
states that, given a closed 2-form $\omega \in \Omega^2 (M)$ with 
integral periods, 
there exists a principal $S^1$-bundle on $M$ with a connection
whose curvature is equal to $\omega$. Moreover the set of such
bundles with connections is in bijection with 
the cohomolgy $H^1 (C^{\bullet} (M, \mathbb{R}/\mathbb{Z}))$ of the complex 
of smooth singular cochains with values in $\mathbb{R}/\mathbb{Z}$. 

In this subsection we prove an equivariant version of 
Kostant theorem on a Lie groupoid $M \leftleftarrows \Gamma$.
Denote 
$\Omega_{0}^2 (M \leftleftarrows \Gamma ) \subset 
\Omega^2 (M \leftleftarrows \Gamma )  = 
\Omega^2 (M) \times \Omega^1 (\Gamma) \times \Omega^0 (\Gamma_2)$
the group of closed 2-forms on $M \leftleftarrows \Gamma$
which are integral (i.e. have integral periods or, equivalently, 
represent images of integral cohomology classes in de Rham cohomology)
and have vanishing $\Omega^0 ( \Gamma_2 )$-component. 
Suppose 
$\Theta = (\omega , \alpha, 0) \in \Omega^2_0 \subset 
\Omega^2 (M) \times \Omega^1 (\Gamma) \times \Omega^0 (\Gamma_2)$.
Since $\Theta$ represents the image of an integral cohomology class
in real smooth singular cohomology 
there exists an integral cochain 
$(c , b , a ) \in C^2 ( M \leftleftarrows \Gamma , \mathbb{Z}) =
C^2 (M , \mathbb{Z}) \times C^1 (\Gamma , \mathbb{Z} )  
\times C^0 (\Gamma_2 , \mathbb{Z})$ 
and a real cochain 
$(h , f ) \in C^1 ( M \leftleftarrows \Gamma , \mathbb{R}) =
C^1 (M , \mathbb{R}) \times C^0 (\Gamma , \mathbb{R})$
such that 
$d_{\mathrm{tot}} ( h, f )  = (\omega , \alpha , 0)  - 
(c , b , a )$, where $d_{\mathrm{tot}}$ is the total differential
in the double complex of smooth singular chains on 
$M \leftleftarrows \Gamma$. This implies that the projection map
\begin{gather}\nonumber
\eta: 
H^2 (DC_{2-1}^{\bullet} ( M \leftleftarrows \Gamma )) \rightarrow 
\Omega^2_0 (M \leftleftarrows \Gamma )
\\\nonumber
\eta \ ( \, [ ( (c ,h ,\omega ), 
(b, f , \alpha), (a , 0, 0) )] \, ) = (\omega , \alpha, 0)
\end{gather}
is surjective, and it is easy to see that the kernel of $\eta$ is
isomorphic 
(by the map $
\bigl[ \bigl( (c ,h , 0 ), 
(b, f , 0), (a , 0, 0) \bigr) \bigr] \mapsto
[(h \ \mathrm{mod} \, \mathbb{Z}, \, f \ \mathrm{mod} \, \mathbb{Z})]$
to the cohomology group 
$H^1 (C^\bullet (M \leftleftarrows \Gamma, \mathbb{R}/\mathbb{Z}))$ 
of smooth singular 1-cochains with values in 
$\mathbb{R}/\mathbb{Z}$. Hence we obtain a short exact sequence
\begin{equation}\label{KostantSequence}
0
\rightarrow 
H^1 (C^\bullet (M \leftleftarrows \Gamma , \mathbb{R}/\mathbb{Z}))
\rightarrow
H^2 (DC_{2-1}^{\bullet} ( M \leftleftarrows \Gamma ) )
\rightarrow 
\Omega^2_0 (M \leftleftarrows \Gamma )
\rightarrow
0 \ .
\end{equation}
This sequence (in the manifold case) first appeared in the original
Cheeger and Simons paper \cite{CheegerSimons1985}.
Since $H^2 (DC_{2-1}^{\bullet} ( M \leftleftarrows \Gamma ))$
classifies equivariant principal $S^1$-bundles
on $M \leftleftarrows \Gamma$
with arbitrary connections (cf. Theorem \ref{InvPreq}), 
we can interpret \eqref{KostantSequence} as follows:

\begin{theorem}
Let $\Theta = (\omega , \alpha, 0) 
\in \Omega^2 ( M \leftleftarrows \Gamma ) =
\Omega^2 (M) \times \Omega^1 (\Gamma) \times \Omega^0 (\Gamma_2)$
be an integral closed 2-form on a Lie groupoid $M \leftleftarrows \Gamma$. 
Then there exists a $\Gamma$-equivariant principal $S^1$-bundle 
$\pi: P \rightarrow M$ with a connection $A \in \Omega^1 (P)$ such that 
$\omega$ is the curvature of the connection $A$ and
$\pi^* \alpha = \partial_1^* A - \partial_0^* A$.
Moreover the set of isomorphism 
classes of such pairs $(P,A)$ is in bijection with 
$H^1 (C^\bullet( M \leftleftarrows \Gamma  , \mathbb{R}/\mathbb{Z}))$.
\end{theorem}

\noindent
For proper groupoids this theorem was proved (by a different argument) in 
\cite{BehrendXu2006}. 

\section{Invariant connections and the moment map}  
\label{InvariantSection}

\noindent
In this section we identify differential characters corresponding to 
invariant connections.

\subsection{The moment map}\label{MomentDef}
We again think of an equivariant principal $S^1$-bundle on
a groupoid $M \leftleftarrows \Gamma$ as a simplicial bundle
\begin{equation}\nonumber
\xymatrix{ 
\ar[d]_{\pi^{_M}} {P} &
\ar@<-0.0ex>[l]
\ar@<-0.8ex>[l]  
\ar[d]^{\pi^{_\Gamma}} {P_1} &
\ar@<0.9ex>[l] \ar@<-0.1ex>[l] \ar@<-1.1ex>[l] 
\ar[d]^{\pi^{_{\Gamma_2}}}
{P_2} &
\ar@<1.3ex>[l]
\ar@<0.3ex>[l]
\ar@<-0.7ex>[l]
\ar@<-1.7ex>[l]  
{\cdots} \\
{M} &
\ar@<0.0ex>[l]
\ar@<-0.8ex>[l]  
{\Gamma} &
\ar@<0.9ex>[l]
\ar@<-0.1ex>[l]
\ar@<-1.1ex>[l]  
{\Gamma_2} &
\ar@<1.3ex>[l]
\ar@<0.3ex>[l]
\ar@<-0.7ex>[l]
\ar@<-1.7ex>[l]  
{\cdots}}
\end{equation}

Recall (see e.g. \cite{MoerdijkMrcun2003}) that the Lie algebroid
$\mathfrak{P}$ of the Lie groupoid $P \leftleftarrows P_1$ is a
vector bundle over $P$ whose fiber at $p \in P$ is equal
to $\mathrm{Ker} \, ds \subset T_pP_1$ 
(here we consider $P$ as a subset of $P_1$
via the unit map). Another way to think of sections of $\mathfrak{P}$
is as right-invariant vector fields 
in $\mathrm{Ker} \, ds \subset TP_1$.
The Lie algebroid $\mathfrak{P}$ comes equipped with a bundle map
$\rho : = dt : \mathfrak{P} \rightarrow TP$ called the anchor.

Let $\mathfrak{P}^*$ be the dual bundle of $\mathfrak{P}$. Given
a connection $A \in \Omega^1 (P)$ we define
the \emph{moment map} $\mu^P_A$ as a section of $\mathfrak{P}^*$ 
given by:
\begin{equation}\nonumber
\mu^P_A ( X ) = A ( \rho (X) ) 
\end{equation}
for a section $X$ of $\mathfrak{P}$. One can rewrite this definition as
\begin{multline}\label{MuDA}
\mu^P_A ( X ) = A ( dt (X) ) = A ( dt (X) - ds (X)) 
= \\ = 
(dt^* (A) - ds^* (A) ) (X) = (\delta A) (X) \ .
\end{multline}
Here $\delta$ is, as usual, the simplicial differential, 
and we think of $X$ as a right-invariant section of 
$\mathrm{Ker} \  ds \subset TP_1$
and of $\mu^P_A (X)$ as a right-invariant function on $P_1$, 
hence a function on $P$. 

Since a connection on and all the structure maps of an 
simplicial principal $S^1$-bundle 
are preserved under the action of $S^1$, the moment map
$\mu^P_A$ is a pullback under $\pi^M$ of a section $\mu_A^M$ of 
the dual bundle  $\mathfrak{G}^*$ of the Lie algebroid $\mathfrak{G}$
of the base groupoid $M \leftleftarrows \Gamma$. Without much
risk of confusion we call both 
$\mu^M_A$ and $\mu^P_A$ the moment maps and often drop the top index.
By definition a connection $A$ is \emph{horizontal} iff $\mu_A = 0$.

For example, the Lie algebroid 
of an action groupoid $P \leftleftarrows G \times P$
is the trivial bundle $P \times \mathfrak{g}$
and the moment map is a map 
$\mu: P \rightarrow \mathfrak{g}^*$
given by $<\! \mu , X \!\! > = i_{\rho(X)} A$, $X \in \mathfrak{g}$. 

\subsection{Invariant connections}\label{InvariantConnections}
We continue with the setup of the previous subsection. Our goal is
to define what it means for a connection to be $P_1$-invariant. 
In the case of a global action 
$P \leftleftarrows G \times P$ a connection $A \in \Omega^1 (P)$ 
is $G$-invariant iff $g^* A = A$ for any $g \in G$. In the groupoid language
an element $g \in G$ should be thought of as a constant section $\Sigma$
of the source map $s$, and the condition $g^* A = A$ can be written as 
\begin{equation}\nonumber
(t^* A -s^* A) |_{_{T\Sigma}} = 0  
\end{equation}
or
\begin{equation}\label{TSigmaZero}
(\delta A) |_{_{T\Sigma}} = 0  
\end{equation}
For a general (non-action) groupoid there is no natural choice of the
section $\Sigma$ through a point $p \in P_1$. One can fix a local regular
foliation $\Sigma$ of $P_1$ by bisections 
(submanifolds transverse to both $s$ and $t$) and say that $A$ is 
$\Sigma$-invariant if \eqref{TSigmaZero} holds. This condition depends 
on the choice of $\Sigma$. 

Note that we have a decomposition 
\begin{equation}\label{TP1Oplus}
T P_1 = T \Sigma \oplus \mathrm{Ker} \  ds
\end{equation}
Hence the two equations \eqref{TSigmaZero} and
\eqref{MuDA} completely determine $\delta A$. In particular,
if the moment map $\mu$ vanishes at some point $p \in P_1$ 
then the condition 
\eqref{TSigmaZero} at the point $p$ does not depend on the choice 
of $\Sigma$. If $\mu$ vanishes identically than the choice of 
$\Sigma$ is irrelevant and the connection is invariant 
(for some $\Sigma$) iff it is basic ($\delta A = 0$).
In other words, as in the global quotient case, the connection is
basic iff it is horizontal and invariant. However for a general 
groupoid the second condition is natural only if the first one is 
satisfied.

\subsection{Prequantization on global quotients}
\label{GlobalQPreq}
We would like to characterize bundles with invariant connections via
their differential characters. Let us restrict ourselves to a global 
quotient case $M \leftleftarrows G \times M$ to avoid choice
of a local foliation $\Sigma$. Given an
equivariant bundle $P$ with an invariant connection $A$ we apply
the functor $\mathrm{DCh}$ (cf. Theorem \ref{InvPreq})
to obtain a character
\begin{equation}\nonumber
\mathrm{DCh} \bigl( (P, A) \bigr) = 
\bigl( (c, h, \omega), [(b, f, \alpha )] \bigr)
\end{equation}
where $\pi^* \alpha = \delta A \in \Omega^1 ( G \times P )$.
Then, according to the previous subsection, $G$-invariance 
of the connection $A$ is equivalent to the following vanishing 
condition on the form $\alpha \in \Omega^1 (G \times M)$:
\begin{equation}\nonumber
i_v \, \alpha = 0
\end{equation}
for any vector field $v \in TM \subset T(G \times M) = TG \oplus TM$.
This means that 
\begin{equation}
\alpha = <\mu,  dg \, g^{-1}> \ ,
\end{equation}
where $dg \, g^{-1}$ is the right-invariant Maurer-Cartan form on $G$
(pulled back to $G \times M$)
and $\mu: M \rightarrow \mathfrak{g}^*$ is the moment map
(pulled back to $G \times M$).

Putting everything together we obtain the following description of
equivariant bundles with invariant connections.
\begin{theorem}\label{GlobalQPreqThm}
Suppose a Lie group $G$ acts on a manifold $M$.
Then the prequantization functor
\begin{equation}\nonumber
\mathrm{Preq}: \ \mathcal{DC}^2_{2-1} (M \leftleftarrows G \times M)
\xrightarrow{\sim} d\mathcal{BS}^1 (M \leftleftarrows G \times M)
\end{equation}
restricts to an equivalence
\begin{equation}\nonumber
\mathrm{Preq}: \ \mathcal{DC}^2_{\mathrm{inv}} (M \leftleftarrows G \times M)
\xrightarrow{\sim} 
d\mathcal{BS}^1_{\mathrm{inv}} (M \leftleftarrows G \times M) \ ,
\end{equation}
where 
$d\mathcal{BS}^1_{\mathrm{inv}} (M \leftleftarrows G \times M)$ is the 
category of $G$-equivariant bundles with $G$-invariant connections
and 
$\mathcal{DC}^2_{\mathrm{inv}} (M \leftleftarrows G \times M)$
is the category of differential characters of the form
$((c, h, \omega), [(b, f, <\mu,  dg \, g^{-1}> )])$.
\end{theorem}

To summarize we have two towers of full subcategories 
related by prequantization/character equivalence:

\begin{equation}\nonumber
\xymatrix{
{\mathcal{DC}^2_{2-1} (M \leftleftarrows G \times M)}
\ar[r]^{\mathrm{Preq}}
&
{d\mathcal{BS}^1 (M \leftleftarrows G \times M)}
\\
{\mathcal{DC}^2_{\mathrm{inv}} (M \leftleftarrows G \times M)} 
\ar[r]^{\mathrm{Preq}}
\ar@{^{(}->}[u]
&
{d\mathcal{BS}^1_{\mathrm{inv}} (M \leftleftarrows G \times M)}
\ar@{^{(}->}[u]
\\
{\mathcal{DC}^2_{2} (M \leftleftarrows G \times M)} 
\ar[r]^{\mathrm{Preq}}
\ar@{^{(}->}[u]
&
{\mathcal{DBS}^1 (M \leftleftarrows G \times M)}
\ar@{^{(}->}[u]
}
\end{equation}
The bottom row consists of equivariant bundles with basic
connections and characters with vanishing $\alpha$-component;
the middle row bundles have invariant connections and the characters have
$\alpha = <\! \mu, dg \, g^{-1} \! >$; finally the top row does not involve any
conditions on the connection or $\alpha$-component of the character.

\section{Reduction}
\label{ReductionSection}

The idea of reduction is that after restriction to the zero level of
the moment map an invariant connection becomes basic and thus defines
a connection on a bundle on the quotient stack. 
From our point of view it is more natural to consider categories of bundles
rather than a single bundle. Hence we start with the ``zero level set'' 
and consider all bundles with the moment map vanishing on that 
manifold.

\subsection{General Reduction}
Let $M \leftleftarrows \Gamma$ be a Lie groupoid,
$N$ a submanifold of $M$ such that $s^{-1} N = t^{-1} N$, 
$i : N \rightarrow M$ the inclusion map, and 
$N \leftleftarrows \Gamma_{_N}$ the corresponding full subgroupoid of 
$M \leftleftarrows \Gamma$ 
(i.e. $\Gamma_{_N} = \Gamma \times_{_s M _i} N = \Gamma \times_{_t M _i} N$). 

We denote by $d\mathcal{BS}^1_{N} (M \leftleftarrows \Gamma)$ 
the full subcategory of 
$d\mathcal{BS}^1 (M \leftleftarrows \Gamma)$ consisting of
equivariant principal $S^1$-bundles with connections $(P,A)$ such that 
the restriction $i^* (P,A)$ of $(P,A)$ to $N$ is an equivariant
bundle with a basic connection (i.e. such that $i^* s^* A = i^* t^* A$). 
As explained in the previous section, this condition can be split into two:
\begin{enumerate}
\item
the moment map of $A$ vanishes on $N$: $\mu_A |_{_N} = 0$;
\item
$i^* A$ is $\Gamma_{_N}$-invariant (with respect to any bisection).
\end{enumerate}
Similarly we define a full subcategory 
$\mathcal{DC}^2_{N} (M \leftleftarrows \Gamma)$ 
of $\mathcal{DC}^2_{2-1} (M \leftleftarrows \Gamma)$ consisting of
differential characters $((c,h,\omega ),[(b,f, \alpha )])$ such that
$\alpha |_{_{\Gamma_{_N}}} = 0$.

It follows from Theorem \ref{InvPreq} that 
the prequantization functor
\begin{equation}\nonumber
\mathrm{Preq}: \mathcal{DC}^2_{2-1} (M \leftleftarrows \Gamma)
\xrightarrow{\sim} d\mathcal{BS}^1 (M \leftleftarrows \Gamma)
\end{equation}
restricts to an equivalence
\begin{equation}\nonumber
\mathrm{Preq}: \mathcal{DC}^2_{N} (M \leftleftarrows \Gamma)
\xrightarrow{\sim} 
d\mathcal{BS}_{N}^1 (M \leftleftarrows \Gamma) \ .
\end{equation}
The restriction 
$i^* d\mathcal{BS}^1_{N} (M \leftleftarrows \Gamma)$ of the
category \mbox{$d\mathcal{BS}^1_{N} (M \leftleftarrows \Gamma)$} to 
the submanifold $N$ is equivalent to the category 
$\mathcal{DBS}^1 (N \leftleftarrows \Gamma|_{_N})$
of equivariant bundles with basic connections on $N$, and
similarly $i^* \mathcal{DC}^2_{N} (M \leftleftarrows \Gamma) \simeq
\mathcal{DC}^2_2 (N \leftleftarrows \Gamma_{_N})$. 
Hence we obtain the following result:

\begin{theorem}Prequantization commutes with reduction, i.e. the following
diagram commutes up to natural transformations:
\begin{equation}\nonumber
\xymatrix{
{\mathcal{DC}^2_{2-1} (M \leftleftarrows \Gamma)}
\ar[r]^{\mathrm{Preq}}
&
{d\mathcal{BS}^1 (M \leftleftarrows \Gamma)}
\\
{\mathcal{DC}^2_{N} (M \leftleftarrows \Gamma)} 
\ar@{^{(}->}[u]
\ar[r]^{\mathrm{Preq}} 
\ar[d]^{\textstyle i^*}
&
{d\mathcal{BS}^1_{N} (M \leftleftarrows \Gamma)}
\ar@{^{(}->}[u]
\ar[d]^{\textstyle i^*}
\\
{\mathcal{DC}^2_{2} (N \leftleftarrows \Gamma_{_N})} 
\ar[r]^{\mathrm{Preq}}
\ar[d]^{\textstyle \wr}
&
{\mathcal{DBS}^1 (N \leftleftarrows \Gamma_{_N})}
\ar[d]^{\textstyle \wr}
\\
{\mathcal{DC}^2_{2} ([ N / \Gamma_{_N} ])} 
\ar[r]^{\mathrm{Preq}}
&
{\mathcal{DBS}^1 ([ N / \Gamma_{_N} ])}
}
\end{equation}
\end{theorem}
Note that prequantization functors in the above
diagram are defined differently for each row
(the top two rows are constructed in the present paper while
the bottom two are from \cite{LermanMalkinPreq}). The meaning
of the theorem is that these definitions are compatible.
 
\subsection{Reduction for action groupoids}
\label{ActionReduction}
On an action groupoid
\mbox{$M \leftleftarrows G \times M$} the reduction can be performed in
a more familiar way: first consider subcategories 
\mbox{$d\mathcal{BS}^1_{\mathrm{inv}} (M \leftleftarrows G \times M)$}, 
$\mathcal{DC}^2_{\mathrm{inv}} (M \leftleftarrows G \times M)$ 
of invariant connection/characters and then further subcategories
\mbox{$d\mathcal{BS}^1_{\mathrm{inv}, N} (M \leftleftarrows G \times M)$}, 
$\mathcal{DC}^2_{\mathrm{inv}, N} (M \leftleftarrows G \times M)$ 
of invariant connection with the moment map vanishing on 
a given $G$-stable submanifold $N$. The categories 
\mbox{$d\mathcal{BS}^1_{\mathrm{inv}, N} (M \leftleftarrows G \times M)$},
\mbox{$\mathcal{DC}^2_{\mathrm{inv}, N} (M \leftleftarrows G \times M)$} 
are in general proper subcategories of 
$d\mathcal{BS}^1_{N} (M \leftleftarrows G \times M)$,
$\mathcal{DC}^2_{N} (M \leftleftarrows G \times M)$ introduced in the
previous subsection
since in the latter ones we impose invariance condition only on 
the submanifold $N$. 

This version of reduction also commutes with prequantization.

\begin{theorem}\label{ActionReductionThm}
The diagram
\begin{equation}\label{ReductionDiagram}
\xymatrix{
{\mathcal{DC}^2_{2-1} (M \leftleftarrows G \times M)}
\ar[r]^{\mathrm{Preq}}
&
{d\mathcal{BS}^1 (M \leftleftarrows G \times M)}
\\
{\mathcal{DC}^2_{\mathrm{inv}} (M \leftleftarrows G \times M)} 
\ar@{^{(}->}[u]
\ar[r]^{\mathrm{Preq}} 
&
{d\mathcal{BS}^1_{\mathrm{inv}} (M \leftleftarrows G \times M)}
\ar@{^{(}->}[u]
\\
{\mathcal{DC}^2_{\mathrm{inv}, N} (M \leftleftarrows G \times M)} 
\ar@{^{(}->}[u]
\ar[r]^{\mathrm{Preq}} 
\ar[d]^{\textstyle i^*}
&
{d\mathcal{BS}^1_{\mathrm{inv}, N} (M \leftleftarrows G \times M)}
\ar@{^{(}->}[u]
\ar[d]^{\textstyle i^*}
\\
{\mathcal{DC}^2_{2} (N \leftleftarrows G \times N)} 
\ar[r]^{\mathrm{Preq}}
\ar[d]^{\textstyle \wr}
&
{\mathcal{DBS}^1 (N \leftleftarrows G \times N)}
\ar[d]^{\textstyle \wr}
\\
{\mathcal{DC}^2_{2} ([ N / G ])} 
\ar[r]^{\mathrm{Preq}}
&
{\mathcal{DBS}^1 ([ N / G ])}
}
\end{equation}
commutes up to natural transformations. 
\end{theorem}

The proof is again a 
simple application of Theorem \ref{InvPreq} (more precisely, 
of the explicit relation between $\delta A$ and the 
$\alpha$-part of a differential character).

\section{Differential Characters as Characters}
\label{CharactersSection} 

In this section we provide an
explicit description of the actual differential characters on
groupoids, i.e. of the isomorphism classes of objects of 
$\mathcal{DC}^2_{2-1} (M \leftleftarrows \Gamma)$
or, in other words, of the second total cohomology group of the complex
$DC^{\bullet}_{2-1} (M \leftleftarrows \Gamma)$.

\subsection{Differential characters on manifolds and groupoids}
Let us first recall the description of Cheeger-Simons differential characters
in the case of a manifold $M$
(cf. \citelist{\cite{CheegerSimons1985} \cite{HopkinsSinger2005}}).
We are interested in the cohomology group
\begin{multline}\nonumber
H^2(DC^{\bullet}_{2} (M) 
= \\ =  
\frac{
\{ 
(c, h, \omega) 
\in C^2 (M, \mathbb{Z}) \times C^1 (M, \mathbb{R}) \times \Omega^2(M) 
\ | \ dc = 0, \ \omega - c - dh = 0, \ d \omega = 0
\}
}{
(c, h, \omega) \sim (c + da , h - a - dt, \omega) 
\ , \  
(a, t) \in C^1 (M, \mathbb{Z}) \times C^0 (M, \mathbb{R})
}
\end{multline}
It is easy to see that equivalence (cohomology) class
of $(c,h,\omega)$ is completely determined by $\omega$ 
and the values of \mbox{$h$ mod $\mathbb{Z}$} on 
smooth 1-cycles. To make a precise statement let us introduce
the group of differential characters on $M$. A differential character
is a pair $(\omega, \chi )$, where $\omega \in \Omega^2 (M)$, $d\omega = 0$,
and $\chi$ is a character $Z_1 (M) \rightarrow \mathbb{R}/\mathbb{Z}$ 
of the group of smooth 1-cycles $Z_1 (M)$. 
This pair should satisfy the following condition:
\begin{equation}\label{CInt}
\chi ( \partial S ) = \int_S \omega \quad \mathrm{mod} \ \mathbb{Z}
\end{equation}
for any smooth 2-chain $S \in C_2 (M, \mathbb{Z})$. The map
\begin{equation}\nonumber
[(c, h, \omega )] \mapsto 
(\omega , \ h|_{_{Z_1 (M)}} \ \mathrm{mod} \ \mathbb{Z})
\end{equation}
provides an isomorphism between $H^2(DC^{\bullet}_{2} (M))$ and the
group of differential characters. Condition \eqref{CInt} ensures
that $c=\omega - dh$ is an integral cochain. 

Let us turn to the groupoid case. We would like to have an explicit
description of the second total cohomology group of the complex
$DC^{\bullet}_{2-1} (M \leftleftarrows \Gamma)$. 
An element of $H^2 (DC^{\bullet}_{2-1} (M \leftleftarrows \Gamma))$
is a closed cochain 
$((c, h, \omega), [(b, f, \alpha)])$ up to addition of
the total differential of $[(a,t,0)]$, see \ref{CharDef} 
for notation. The cochain being closed means
\begin{align}\label{CharD}
d(c, h, \omega) &= 0 \\ \label{CharDeltaD}
\delta (c, h, \omega ) &= d [(b, f, \alpha)] \\ \label{CharDelta}
\delta [(b, f, \alpha)] &= 0 
\end{align}

Similar to the manifold case $c$ and $b$ are determined by
$h$, $f$, $\omega$ and $\alpha$:
\begin{align}\label{DefB}
b &= \alpha - d f - \delta h \\ \label{DefC}
c &= \omega -dh
\end{align}

Turning to differential form components, we put 
\begin{equation}
\Theta = (\omega, \alpha, 0) \in 
\Omega^2 (M)  \oplus \Omega^1 (\Gamma )  \oplus \Omega^0 (\Gamma_2 ) 
= \Omega^2 (M \leftleftarrows \Gamma ) \ .
\end{equation}
Then $\Theta $ is a closed 2-form on $M \leftleftarrows \Gamma $ 
with vanishing $\Omega^0 (\Gamma_2)$-component.

Finally the pair $(h,f)$ is determined up to coboundary 
(total differential of $[(a,t,0)]$) by
its values mod $\mathbb{Z}$ on smooth 1-cycles 
$Z_1 (M \leftleftarrows \Gamma )$. Recall that a cycle
(generator of $Z_1$) on a groupoid looks like this
(see for example \cite{Behrend2004}):
\begin{equation}\nonumber
\xymatrix{
& {} \ar@{-->}[r]^{g_1} & {} \\
{} \ar@{->}[ur]^{\gamma_1} & & & {} \ar@{<-}[ul]_{\gamma_2} 
\ar@{-->}[d]^{g_2}\\
{} \ar@{-->}[u]^{g_n} & & & {} 
  \\
& {}  & \ar@{.}@/^1pc/[llu]{} \ar@{<-}[ur]+D*{}_{\gamma_3}
}
\end{equation}
Here $\gamma_i : [0,1] \rightarrow M$ are smooth paths in $M$ and
$g_i \in \Gamma$ are arrows connecting end-points of these paths.
This is a particular example of a smooth singular 1-chain on the 
groupoid and so it makes sense to evaluate the cochain $(h,f)$ on this chain:
$\sum h( \gamma_i ) + \sum f( g_i ) \in \mathbb{R}$. 
Taking values mod $\mathbb{Z}$ we get a character
\begin{equation}\nonumber
\Psi = (h,f)|_{_{Z_1 (M \leftleftarrows \Gamma )}} \ 
\mathrm{mod} \ \mathbb{Z} \quad : \quad 
Z_1 (M \leftleftarrows \Gamma ) 
\rightarrow \mathbb{R}/\mathbb{Z} \ .
\end{equation}

The pair $(\Theta, \Psi )$ uniquely determine the cohomology class
of $((c, h, \omega), [(b, f, \alpha)])$. Conversely, a pair
$\bigl( \Theta \in \Omega^2 (M \leftleftarrows \Gamma ), \ 
\Psi :  Z_1 ( M \leftleftarrows \Gamma ) 
\rightarrow \mathbb{R} / \mathbb{Z} \bigr)$
corresponds to a cohomology class in 
$H^2 (DC^{\bullet}_{2-1} (M \leftleftarrows \Gamma))$
iff it satisfies the following conditions:
\begin{enumerate}
\item 
\label{FirstCond}
$\Theta$ is a closed 2-form on
$(M \leftleftarrows \Gamma )$ with vanishing
$\Omega^0 (\Gamma_2)$-component. This ensures that the
third components of the relations 
\eqref{CharD} - \eqref{CharDelta} hold true.
\item 
\label{CharCond}
$\Psi (g_1 g_2 g_3) = 0$ for $g_1, g_2, g_3 \in \Gamma$
such that $g_1 g_2 g_3 = id_p$, $p \in M$. 
Note that such a triple $g_1$, $ g_2$,  $g_3$
is a cycle in $Z_1 (M \leftleftarrows \Gamma )$: 
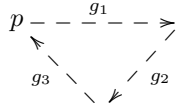
\begin{equation}\nonumber
\xymatrix{
{p} \ar@{-->}[rr]^{g_1} & & {} \ar@{-->}[ld]^{g_2} \\
& {} \ar@{-->}[ul]^{g_3} & {}
}
\end{equation}
This condition ensures that $\delta f =0$. 
\item
\label{GroupLoopCond}
For any smooth path $\eta: [0,1] \rightarrow \Gamma$ one has
\begin{equation}\nonumber
\Psi (\partial \eta) =  
\int_{\eta} \alpha \quad \mathrm{mod} \ \mathbb{Z} \ ,
\end{equation}
where $\partial \eta$ is the total boundary of $\eta$, i.e. 
the following cycle in \mbox{$Z_1 (M \leftleftarrows \Gamma )$}:
\begin{equation}\nonumber
\xymatrix{
{} \ar@{-->}[r]^{\eta (0)} & {} \ar@{->}[d]^{t(\eta )} \\
{} \ar@{->}[u]^{-s(\eta )} \ar@{<--}[r]_{\eta (1)^{-1}} & {}
}
\end{equation}
This condition ensures that $b (\eta )$ defined by \eqref{DefB} belongs
to $\mathbb{Z}$.
\item\label{LastCond}
For any smooth 2-chain $S$ on $M$ one has
\begin{equation}\nonumber
\Psi ( \partial S ) = 
\int_S \omega \quad \mathrm{mod} \ \mathbb{Z} \ .
\end{equation}
Note that $\partial S$ is a 1-cycle on $M$ but we think of it
as an element of $Z_1 (M \leftleftarrows \Gamma )$ 
\begin{equation}\nonumber
\qquad \quad
\xymatrix{
{} \ar `r[r] `[d]^{\partial S}_{\textstyle S \quad {}}  `[d] []  & {} 
%\ar@{->}[d]^{\partial S} 
\\
{} 
%\ar@{->}[u]_{\textstyle S} 
%\ar@{<-}[r] 
& 
{}
}
\end{equation}
This condition ensures that $c (S)$ defined by \eqref{DefC} belongs
to $\mathbb{Z}$.
\end{enumerate}

\noindent
Let us summarize the above discussion as the following version of Theorem
\ref{InvPreq}.

\begin{theorem} Isomorphism classes of equivariant $S^1$-bundles 
with (not necessarily basic) connections on a 
groupoid $M \leftleftarrows \Gamma$ are in bijection
with differential characters, i.e. pairs $(\Theta , \Psi)$, where 
$\Theta \in \Omega^2 ( M \leftleftarrows \Gamma )$ and
$\Psi$ is a character 
\mbox{$Z_1 (M \leftleftarrows \Gamma ) \rightarrow \mathbb{R}/\mathbb{Z}$},
satisfying conditions \eqref{FirstCond}--\eqref{LastCond} above. 

Characters $(\Theta , \Psi)$ with vanishing 
$\Omega^1 (\Gamma )$-component of $\Theta$
correspond to equivariant bundles with basic connections,
i.e. bundles with connections on the quotient stack.
\end{theorem} 

\noindent
\emph{Remark.} 
Note that $(\Theta |_{_M} , \Psi |_{_{Z_1(M)}})$ defines a differential 
character on $M$ corresponding to an $S^1$-bundle with connection on $M$. 
The above Theorem says that equivariant structures 
on this bundle are in 1-1 correspondence
with equivariant extensions of its character.

\subsection{Special cases}
Because of conditions \eqref{CharCond} -- \eqref{LastCond} 
the character $\Psi$ is determined by its value on the following
classes of 1-cycles for each point $p \in M$: 
(1) generators of the fundamental group of
$M$ with the base point $p$, (2) representatives of the connected 
components of the inertia group $I_p = s^{-1} p \cap t^{-1} p$,
and (3) cycles of the form
\begin{equation}\nonumber
\xymatrix{{p} \ar@{<--}@/_/[rr] & &{} \ar@{<-}@/_/[ll]} \qquad  ,
\end{equation}
where as before the solid line is a path in $M$ and the broken line is
an arrow in $\Gamma$. Moreover it is enough to consider one base 
point $p$ in each connected (via sequences of paths in $M$ and
arrows in $\Gamma$) component of $M$. This description of $\Psi$, though 
economical, is not very natural because homotopy of paths/points involves
integration of $\omega$ and/or $\alpha$. However in some cases it 
can clarify (and simplify) things substantially. Here are some examples:

\begin{enumerate}
\item
If $\Theta = 0$ then $\Psi$ is just a homomorphism of the 
product of the (suitably defined) fundamental groups 
of connected components of the stack $[M / \Gamma]$ into
$\mathbb{R}/\mathbb{Z}$, i.e. a local system on $[M / \Gamma]$.
\item
If the groupoid is source-connected (e.g. it is the action 
groupoid of a connected group action)  
then $\Psi$ is determined by its values on $Z_1 (M)$ 
(or just on generators of fundamental groups of connected components 
of $M$).
\item
For any point $p \in M$ the character $\Psi$ restricts to a homomorphism 
$I_p \rightarrow \mathbb{R}/\mathbb{Z}$, where $I_p$ is the inertia 
group at $p$. Moreover this homomorphism is determined by its values on 
representatives of the connected components of $I_p$.
Note however that, unless the quotient of $I_p$ by its identity component
$I_p^e$ has a section or the form $\alpha$ vanishes, 
$\Psi$ does not lift to a homomorphism 
$I_p/I_p^e \rightarrow \mathbb{R}/\mathbb{Z}$.
\item
In the case of a finite group $G$ acting on a point, a differential
character is just a homomorphism $G \rightarrow \mathbb{R}/\mathbb{Z}$.
\end{enumerate}

\subsection{Invariant differential characters on an action groupoid}
\label{ActionChar}
In the case of an action groupoid  $M \leftleftarrows G \times M$
we have a notion of invariant connections (cf. subsection \ref{GlobalQPreq})
which fit between arbitrary and basic ones. 
We would like to describe the corresponding differential characters. 
Invariance of a connection corresponds to vanishing of 
the form $\alpha \in \Omega^1 (G \times M)$ on
constant sections of the source map (projection) 
$s : G \times M \rightarrow M$.
Because of relation \eqref{GroupLoopCond} this vanishing condition 
is equivalent to vanishing of 
the character $\Psi$ on cycles of the type
\begin{equation}\nonumber
\xymatrix{
{} \ar@{-->}[r]^{g} & {} \ar@{->}[d]^{g \circ \gamma} \\
{} \ar@{->}[u]^{\gamma} \ar@{-->}[r]_{g} & {}
}
\end{equation}
Here $\gamma : [0,1] \rightarrow M$ is a path in $M$, and we
denote by the same letter $g$ an element of $G$, the corresponding
constant section of $s$, and the corresponding diffeomorphism of $M$.
Taking into account \eqref{CharCond} we see that invariance of
the connection means that the character $\Psi$ is invariant 
with respect to the natural action of $g \in G$ on 
$Z_1 (M \leftleftarrows M \times G)$.

Note that assuming
\eqref{FirstCond}, \eqref{CharCond}, and \eqref{LastCond}, it is
enough to require \eqref{GroupLoopCond} on homotopy representatives
of paths in $G \times M$ and any path in $G \times M$
is homotopic to a composition of a constant section of $s$ and
a 1-parametric subgroup in $G$ acting on a point in $M$.
The above invariance condition takes care of
constant sections of $s$, so let us consider a 1-parametric
subgroup $e^{tX}$, $X\in \mathfrak{g}$. Given a point $p \in M$ the
relation \eqref{GroupLoopCond} applied to the path 
$(e^{tX}, p) \subset G \times M$ reads
\begin{equation}\nonumber
\Psi( \xymatrix{{p} \ar@{<--}@/_/_{e^{-X}}[rr] & &{} 
\ar@{<-}@/_/_{e^{tX}}[ll]}  {e^{X}p} ) = <\mu (p) , X>
\quad \mathrm{mod} \ \mathbb{Z} \ .
\end{equation}
Here  the solid arrow represents the path $e^{tX} p$, $t \in [0,1]$,
in the manifold $M$, the broken arrow the element $(e^{tX}, p)$ of
$G \times M$, and $\mu: M \rightarrow \mathfrak{g}^*$ is the moment map. 

This discussion together with subsection \ref{GlobalQPreq}, 
leads us to the following description of the invariant
characters for an action groupoid.

\begin{theorem}\label{ActionCharThm}
Isomorphism classes of equivariant $S^1$-budles 
with invariant connections on an action groupoid 
$M \leftleftarrows G \times M$ are in bijection
with invariant differential characters, i.e. triples
$(\omega, \mu , \Psi)$, where $\omega \in \Omega^2 ( M )$ is a 2-form, 
$\mu : M \rightarrow \mathfrak{g}^*$ is a (moment) map, and 
$\Psi$ is a character 
$Z_1 (M \leftleftarrows G \times M ) \rightarrow \mathbb{R}/\mathbb{Z}$.
This triple should satisfy the following conditions:
\begin{enumerate}
\item
$\omega$ and $\Psi$ are $G$-invariant, $\mu$ is $G$-equivariant. 
\item\label{InvariantMomentCond}
$i_{\varepsilon_X} \omega = - <d \mu, X>$, where $X \in \mathfrak{g}$ and 
$\varepsilon_X$ is the action vector field of $X$.
\item\label{InvariantOmegaCond}
$(\omega, \Psi |_{_{Z_1(M)}})$ is an (invariant) differential character 
on $M$, i.e. 
\begin{equation}\nonumber
\Psi (\partial S) = \int_S \omega \quad \mathrm{mod} \ \mathbb{Z}
\end{equation}
for any $S \in C_2 (M)$.
\item\label{OneParPsi}
One has
\begin{equation}\nonumber
\Psi( \xymatrix{{p} \ar@{<--}@/_/_{e^{-X}}[rr] & &{} 
\ar@{<-}@/_/_{e^{tX}}[ll]}  {e^{X}p} ) = <\mu (p) , X \! > 
\quad \mathrm{mod} \ \mathbb{Z}
\end{equation}
for any $p \in M$, $X \in \mathfrak{g}$.
\item
$\Psi$ restricts to a character 
$\Psi_p: I_p \rightarrow \mathbb{R}/\mathbb{Z}$ 
of the stabilizer (inertia) group $I_p$ for each point $p \in M$. 
\end{enumerate}
The connection is basic iff $\mu \equiv 0$.
\end{theorem} 

\noindent
Here we rewrote the condition that the form
\mbox{$\bigl( \omega, < \! \mu, dg \, g^{-1} \! \! >, 0 \bigr) \in 
\Omega^2 ( M \leftleftarrows G \times M )$}
is closed in a more familiar form (see for example 
\cite{GuilleminSternberg1999}):
(1) $\omega$ is closed and $G$-invariant,
(2) $\mu: M \rightarrow \mathfrak{g}^*$ is
$G$-equivariant, and  (3) $i_{\varepsilon_X} \omega = - <d \mu, X \! >$.

As for a general groupoid, one only needs to specify
the character $\Psi$ on some elements of 
$Z_1 (M \leftleftarrows G \times M )$.
However a natural choice of these generators depends on the particular 
situation.

\section{Examples}
\label{ExamplesSection}
In this section we sketch several examples of differential characters and
reduction. We leave the details to an interested reader.

\subsection{Example: classifying stack}
Consider a Lie group $G$ acting on a point $\ast$. A differential
character on the groupoid $\ast \leftleftarrows G \times \ast$ is a 
homomorphism $\Psi: G \rightarrow \mathbb{R}/\mathbb{Z}$
together with an element $\mu \in \mathfrak{g}^*$, such that 
$d \Psi = \mu$. Such a pair $(\Psi , \mu)$ is determined by the
character $\Psi$ which has to be smooth. So the differential character
on $\ast \leftleftarrows G \times \ast$ is the same 
as a smooth character of $G$.

\subsection{Example: coadjoint orbits of compact Lie groups}
Let $M$ be a coadjoint orbit of a compact connected Lie group $G$:
$M = G \cdot \lambda$ for $\lambda \in \mathfrak{g}^*$. 
If we assume the moment map
$\mu :M \hookrightarrow \mathfrak{g}^*$ the inclusion then
the rest of a differential character is uniquely determined 
by Theorem \ref{ActionCharThm} as follows. 

Let $\tau : G \rightarrow M$ be given by $\tau (g) =g \cdot \lambda$.
The condition \eqref{InvariantMomentCond} implies
\begin{equation}\nonumber
\tau^* \omega = d <\lambda , g^{-1} dg> \ ,
\end{equation}
where $g^{-1} dg$ is the left-invariant
Maurer-Cartan form. This equality completely determines $\omega$ since
$\tau$ is a surjective submersion.

Now we turn to the character $\Psi$.
Note that $\Psi$ is $G$-invariant and $G$ is compact. 
Therefore $\Psi$ is determined by its values on
$G$-invariant cycles in $Z_1 (G \times M)$ containing $\lambda$. 
Moreover, $G$ being connected, we can assume these cycles to be 
closed paths 
\begin{equation}\nonumber
e^{tX} \cdot \lambda \subset M \ , \quad 0 \leq t \leq 1 \ , 
\quad X \in \mathfrak{g} \ ,
\end{equation} 
produced by actions of 1-parametric subgroups of $G$ on $\lambda$. 
For such cycles equation \eqref{OneParPsi} implies
\begin{equation}\nonumber
\Psi ( \{ e^{tX} \cdot \lambda \} ) = < \! \lambda , X \! > + \, \Psi (e^X) 
\quad \mathrm{mod} \ \mathbb{Z} \ .
\end{equation}
Note that $e^X$ belongs to the stabilizer $I_{\lambda}$
of $\lambda$ in $G$ (in other words to the inertia group at $\lambda$);
we denote $\mathfrak{i}_{\lambda}$ the Lie algebra of $I_{\lambda}$. 
It remains to determine the values of the character 
$\Psi|_{_{I_{\lambda}}}: \ I_{\lambda} \rightarrow \mathbb{R}/\mathbb{Z}$. 
Because of \eqref{OneParPsi}
we have $\mathrm{d} \Psi |_{_{I_{\lambda}}} = \lambda|_{\mathfrak{i}_{\lambda}}$
and $I_{\lambda}$ being connected
this determines $\Psi |_{I_{\lambda}}$, completing construction
of the differential character $( \omega , \mu , \Psi)$.

Of course there are still conditions to be checked for 
$( \omega , \mu , \Psi)$ to be a differential character. 
For example, $\lambda|_{\mathfrak{i}_{\lambda}}$ 
should actually lift to a character of the stabilizer group $I_{\lambda}$. 
In fact (we don't prove it in the present paper),
this is the only condition (for existence of a differential character
on the coadjoint orbit of $\lambda \in \mathfrak{g}^*$). One can say that
differential character on an orbit is induced from a smooth character of
the stabilizer group of a point on the orbit.

Let us remark that usually the integrality condition  
is stated in terms of the form $\omega$ on $M$ and then 
additional conditions on $\lambda$ are imposed to ensure that 
the Lie algebra action on the prequantization bundle lifts
to an action of the group. For example, 
if $G$ is simply connected then the differential character exists
iff $\omega$ is integral.
Our point of view is that it is more natural to emphasise the character
$\Psi$ rather than the form $\omega$.
As an extreme example, if $G=T$ is a torus (compact connected abelian group)
then each coadjoint orbit is a point $\lambda \in \mathfrak{t}^*$, 
hence has unique
(integral) 2-form $\omega$ ($=0$). However only those points
for which $\lambda$ is a differential of a character of $T$ 
are prequantizable according to the above discussion. In fact 
the set of prequantizable orbits in $\mathfrak{t}^*$ is
the lattice of characters of $T$. Moreover, since an
orbit is a point, the differential character is just a
smooth character of $T$.

\subsection{Example:  coadjoint orbits of $SU(2)$ and $S^1$-reduction}
Let $S^2$ be a sphere about the origin in $\mathbb{R}^3$.
We think of $S^2$ as a coadjoint orbit of $SU(2)$ 
in $\mathfrak{su}_2^* \approx \mathbb{R}^3$.
In particular, $SU(2)$ acts on $S^2$ by rotations (via the 
2-fold covering map $SU(2) \rightarrow SO(3, \mathbb{R})$).

It is easy to see that given an integral 2-form $\omega \in \Omega^2 (S^2)$ 
one has unique $SU (2)$-invariant 
differential character $(\omega, \mu, \Psi )$ on $S^2$. 
For example, let us compute the restriction
$\chi : \{ \pm 1 \} \rightarrow \mathbb{R} / \mathbb{Z}$
of the character $\Psi$ to the center 
$\{ \pm 1 \} \subset SU (2)$. Here we think of
the center as a subgroup of the inertia group $I_p$ at a point
$p \in S^2$ (note that $\Psi$ is $SU(2)$-invariant and
hence $\chi$ does not depend on the point $p$). 
Let $X$ be (unique up to sign) element of the Lie algebra $\mathfrak{su}_2$ 
such that $e^{X} = -1 \in SU(2)$, $< \! \mu(p), X \! > = 0$, and
$\gamma (t) = e^{tX} p$, $0 < t < 1$, is a big circle on $S^2$ 
bounding a half-sphere $B \subset S^2$. Then we have
\begin{equation}\nonumber
\chi (-1) = \Psi (e^{X}) 
\stackrel{\eqref{OneParPsi}}{=} 
\Psi (\partial B) 
\stackrel{\eqref{InvariantOmegaCond}}{=} 
\int_B \omega = \frac{1}{2} \int_{S^2} \omega 
\quad \mathrm{mod} \ \mathbb{Z} \ .
\end{equation}
In fact, we did not have to assume that $<\mu (p), X> = 0$. For an
arbitrary $X \in \mathfrak{su}_2$ such that $e^{X} = -1$, one has
\begin{multline}\nonumber
\chi (-1) = \Psi (e^{X})  
\stackrel{\eqref{OneParPsi}}{=} 
\Psi (\partial B) - <\mu (p) , X> 
\stackrel{\eqref{InvariantOmegaCond}}{=} 
\int_B \omega - <\mu (p) , X> 
= \\ =
\frac{1}{2} \int_{S^2} \omega \quad \mathrm{mod} \ \mathbb{Z} \ ,
\end{multline}
where $B$ is either of the two parts of $S^2$ bounded by
the circle $\gamma (t) = e^{tX} p$, $0 < t < 1$, 
and the last equality follows from Archimedes 
(or Duistermaat-Heckman) Theorem.

Let us now turn to the reduction procedure with respect to
the action of a torus $T \approx S^1 \subset SU (2)$. 
We denote by $h$ the restriction of the moment map $\mu$
to the Lie algebra of $T$ (so if $T$ is the group of rotations 
about the vertical axis, then $h$ is the height function on $S^2$
up to scale). According to Section \ref{ReductionSection}
the reduction is the following sequence of functors
(cf. \eqref{ReductionDiagram}):
\begin{multline}\nonumber
{\mathcal{DC}^2_{\mathrm{inv}} 
\bigl( S^2 \leftleftarrows SU(2) \times S^2  \bigr)} 
\rightarrow 
{\mathcal{DC}^2_{\mathrm{inv}} 
\bigl( S^2 \leftleftarrows T \times S^2  \bigr)} 
\rightarrow \\ \rightarrow
{\mathcal{DC}^2_2 
\bigl( h^{-1} (0) \leftleftarrows T \times h^{-1} (0)  \bigr)} 
\xrightarrow{\sim}
{\mathcal{DC}^2_2
\bigl( p \leftleftarrows \{ \pm 1 \} \times p  \bigr)} \ .
\end{multline}
Here $p$ is a point on the circle $h^{-1} (0)$,
the first arrow is the restriction, the second is 
the actual reduction (i.e. restriction to the zero level of
the moment map), and the last (restriction) isomorphism is due to the
fact that $\mathcal{DC}^2_2$ is a stack.

The category ${\mathcal{DC}^2_2
\bigl( p \leftleftarrows \{ \pm 1 \} \times p  \bigr)}$
is just the category of characters (or 1-dim representations) of
the group $\{ \pm 1 \}$. On the other hand, an isomorphism class 
of an object $(c, h, \omega)$ of ${\mathcal{DC}^2_{\mathrm{inv}} 
\bigl( S^2 \leftleftarrows T \times S^2  \bigr)}$
is determined by $\omega$. Since the reduction is just 
the restriction of differential characters we see that it takes $\omega$ to
the character $\chi : \{ \pm 1 \} \rightarrow \mathbb{R} / \mathbb{Z}$ given, 
as above, by 
\begin{equation}\nonumber
\chi (-1) =  \frac{1}{2} \int_{S^2} \omega 
\quad \mathrm{mod} \ \mathbb{Z} \ .
\end{equation}

One of our main motivations for developing the theory of invariant
differential characters was to provide a context for the above formula.
It corresponds to the fact that an irreducible representation of $SU (2)$ has
a $T$-invariant vector iff its highest weight is even.

\end{document}